\documentclass[11pt]{amsart}
\usepackage{graphicx}
\usepackage{amssymb, amsmath}
\newtheorem{theorem}{Theorem}[section]
\newtheorem{corollary}[theorem]{Corollary}

\newtheorem{proposition}[theorem]{Proposition}
\theoremstyle{definition}

\theoremstyle{remark}

\numberwithin{equation}{section}

\newcommand{\Gf}{\mathrm{der}_{\theta, b}(G, \bullet)}
\newcommand{\Gfi}{\mathrm{der}_{\theta_i, b_i}(G_i, \bullet_i)}

\begin{document}
\title{On profinite polyadic groups}
\author{\sc M. Shahryari}

\address{M. Shahryari\\
 Department of Mathematics,  College of Sciences, Sultan Qaboos University, Muscat, Oman}

\author{\sc M. Rostami}
\thanks{{\scriptsize
\hskip -0.4 true cm MSC(2010):  20N15
}}

\address{M. Rostami\\
Department of Pure Mathematics,  Faculty of Mathematical
Sciences, University of Tabriz, Tabriz, Iran }

\email{m.ghalehlar@squ.edu.om}
\email{m.rostami@tabrizu.ac.ir}
\date{\today}

\begin{abstract}
We study the structure of profinite polyadic groups and we prove that a polyadic topological group $(G, f)$ is profinite, if and only if, it is compact, Hausdorff, totally disconnected. More generally, for a pseudo-variety (or a formation) of finite groups $\mathfrak{X}$, we define the class of $\mathfrak{X}$-polyadic groups, and we show that a polyadic group $(G, f)$ is pro-$\mathfrak{X}$, if and only if, it is compact, Hausdorff, totally disconnected and for every open congruence $R$, the quotient $(G/R, f_R)$ is $\mathfrak{X}$-polyadic.
\end{abstract}

\maketitle


\section{Introduction}

In this article, we study the structure of profinite polyadic groups:  polyadic groups which are the  inverse limit of a system of finite polyadic groups.  A polyadic group is a  natural generalization of the concept of group to the case where the binary operation of group replaced with an $n$-ary associative operation, one variable linear equations in which have unique solutions. So, in this article, {\em polyadic group} means an $n$-ary group for a fixed natural number $n\geq 2$. These interesting algebraic objects are introduced by Kasner and D\"ornte (\cite{Kas} and \cite{Dor}) and  studied extensively by Emil Post during the first decades of the last century, \cite{Post}. During decades, many articles have been published on the structure of polyadic groups. Already,  homomorphisms and automorphisms of polyadic groups are studied in  \cite{Khod-Shah}. A characterization of the simple polyadic groups is obtained by them in \cite{Khod-Shah2}. Also,  the representation theory of polyadic groups is studied in  \cite{Dud-Shah} and  the complex characters of finite polyadic groups are also investigated in \cite{Shah2}. The structure of free polyadic groups is determined in \cite{Artam}, \cite{Khod-Shah3}, and \cite{Khod}.

It is easy to define topological polyadic groups, and so, one can ask which topological polyadic groups are profinite. In this paper, we study this problem and as the main result, we prove that a  polyadic topological group $(G, f)$ is profinite, if and only if, it is compact, Hausdorff, totally disconnected. More generally, for a pseudo-variety (formation) of finite groups $\mathfrak{X}$, we define the class of $\mathfrak{X}$-polyadic groups, and we show that a polyadic group $(G, f)$ is pro-$\mathfrak{X}$, if and only if, it is compact, Hausdorff, totally disconnected and for every open congruence $R$, the quotient $(G/R, f_R)$ is $\mathfrak{X}$-polyadic.

\section{Polyadic groups}
A polyadic group is a pair $(G, f)$ where $G$ is a non-empty set and $f:G^n\to G$ is an $n$-ary operation, such that \\

i- the operation is associative, i.e.
$$
f(x_1^{i-1},f(x_i^{n+i-1}),x_{n+i}^{2n-1})=
f(x_1^{j-1},f(x_j^{n+j-1}),x_{n+j}^{2n-1})
$$
for any $1\leq i<j\leq n$ and for all $x_1,\ldots,x_{2n-1}\in G$, and \\

ii- for all $a_1,\ldots,a_n, b\in G$ and $1\leq i\leq
n$, there exists a unique element $x\in G$ such that
$$
f(a_1^{i-1},x,a_{i+1}^n)=b.
$$

Note that, here we use the compact notation $x_i^j$ for every sequence
$$
x_i, x_{i+1}, \ldots, x_j
$$
of elements in $G$, and in the special case when all terms of this sequence are equal to a fixed $x$, we denote
it by $\stackrel{(t)}{x}$, where $t$ is the number of terms.

Clearly, the case $n=2$ is exactly the definition of ordinary groups. During
this article, we assume that $n$ is fixed. Note that an $n$-ary system $(G,f)$ of the form $f(x_1^n)=x_1
x_2\ldots x_nb$, where $(G,\cdot)$ is a group and $b$  a fixed
element belonging to the center of $(G,\cdot)$, is a polyadic group, which is  called {\em $b$-derived} from the group
$(G,\cdot)$ and it is denoted by $\mathrm{der}_b^n(G, \cdot)$. In the case when $b$ is the identity of $(G,\cdot)$, we
say that such a polyadic group is {\em reduced} to the group $(G, \cdot)$ or {\em derived} from $(G, \cdot)$ and we use the notation
$\mathrm{der}^n(G, \cdot)$ for it. For every $n>2$, there
are $n$-ary groups which are not derived from any group. A polyadic group $(G,f)$ is derived from some group if and only if, it contains
an element $a$ (called an {\em $n$-ary identity}) such that
$$
 f(\stackrel{(i-1)}{a},x,\stackrel{(n-i)}{a})=x
$$
holds for all $x\in G$ and for all $i=1,\ldots,n$, see \cite{Dud2}.

From the definition of an $n$-ary group $(G,f)$, we can directly see that for every $x\in G$, there exists only one $y\in G$,  satisfying
the equation
$$
f(\stackrel{(n-1)}{x},y)=x.
$$
This element is called {\em skew} to $x$ and it is denoted by
$\overline{x}$.    As D\"ornte \cite{Dor} proved, the following identities hold for all $\,x,y\in G$, $2\leq i\leq n$,
$$
f(\stackrel{(i-2)}{x},\overline{x},\stackrel{(n-i)}{x},y)=
f(y,\stackrel{(n-i)}{x},\overline{x},\stackrel{(i-2)}{x})=y.
$$
These identities together with the associativity identities, axiomatize the variety of polyadic groups in the algebraic language $(f, ^{-})$.

Suppose $(G, f)$ is a polyadic group and $a\in G$ is a fixed element. Define a binary operation
$$
x\bullet y=f(x,\stackrel{(n-2)}{a},y).
$$
Then $(G, \bullet)$ is an ordinary group, called the {\em retract} of $(G, f)$ over $a$. Such a retract will be denoted by $\mathrm{ret}_a(G,f)$. All
retracts of a polyadic group are isomorphic. The identity of the group $(G,\bullet)$ is $\overline{a}$. One can verify
that the inverse element to $x$ has the form
$$
y=f(\overline{a},\stackrel{(n-3)}{x},\overline{x},\overline{a}).
$$

One of the most fundamental theorems of polyadic group is the following, now known as {\em Hossz\'{u} -Gloskin's theorem}. We will use
it frequently in this article and the reader can use \cite{DG}, \cite{Hos} and \cite{Sok} for detailed discussions.

\begin{theorem}
Let $(G,f)$ be   a polyadic group. Then there exists an ordinary group $(G, \bullet)$,
an automorphism $\theta$ of $(G, \bullet)$ and an element $b\in G$ such that\\

1.  $\theta(b)=b$,\\

2.  $\theta^{n-1}(x)=b x b^{-1}$, for every $x\in G$,\\

3.  $f(x_1^n)=x_1\theta(x_2)\theta^2(x_3)\cdots\theta^{n-1}(x_n)b$, for all $x_1,\ldots,x_n\in G$.

\end{theorem}

According to this theorem, we  use the notation $\Gf$ for $(G,f)$
and we say that $(G,f)$ is $(\theta, b)$-derived from the group $(G,\bullet)$.

There is one more important ordinary group associated to a polyadic group which we call it the {\em Post's cover}. This is the first fundamental theorem concerning polyadic groups. The proof can be find in \cite{Post}.

\begin{theorem}
Let $(G, f)$ be a polyadic group. Then, there exists a unique group $(G^{\ast}, \circ)$  such that \\

1- $G$ is contained in $G^{\ast}$ as a coset of some  normal subgroup $K$.\\

2- $K$ is isomorphic to a retract of $(G, f)$.\\

3- We have $G^{\ast}/K\cong \mathbb{Z}_{n-1}$.\\

4- Inside $G^{\ast}$, for all $x_1, \ldots, x_n\in G$, we have $f(x_1^n)=x_1\circ x_2\circ \cdots \circ x_n$. \\

5- $G^{\ast}$ is generated by $G$.
\end{theorem}

The group $G^{\ast}$ is also universal in the class of all groups having properties 1, 4. More precisely, if $\beta: (G, f)\to \mathrm{der}^n(H, \ast)$ is a polyadic homomorphism, then there exists a unique ordinary homomorphism $h: G^{\ast}\to H$, such that $h_{|_{G}}=\beta$. This universal property characterizes $G^{\ast}$ uniquely. The explicit construction of the Post's cover can be find in \cite{Shah2}.

Finally, we have to mention that the structure of polyadic homomorphisms will be needed in what follows. The reader can see \cite{Khod-Shah} for details.

\begin{theorem}
Suppose $(G,f)=\mathrm{der}_{\theta, b}(G, \cdot)$ and $(H, h)=\mathrm{der}_{\eta,c}(H, \ast)$ are two polyadic groups. Let $\psi: (G,f)\to (H, h)$ be a homomorphism. Then there exists $a\in H$ and an ordinary homomorphism $\phi:(G, \cdot)\to (H, \ast)$, such that $\psi=R(a)\phi$, where $R(a)$ denotes the map $x\mapsto x\ast a$. Further $a$ and $\phi$ satisfy the following conditions;\\
$$
h(\stackrel{(n)}{a})=\phi(b)\ast a\ \ \ and \ \ \
\phi\theta=I_{a}\eta\phi,
$$
where, $I_a$ denotes the inner automorphism $x\mapsto a\ast x\ast
a^{-1}$. Conversely, if $a$ and $\phi$ satisfy the above two conditions, then $\psi=R_a\phi$ is a homomorphism $(G,f)\to  (H,h)$.
\end{theorem}

\section{Profinite polyadic groups}
A profinite polyadic group is the inverse limit of an inverse system of finite polyadic groups. More precisely, let $(I, \leq)$ be a directed set and suppose $\{ (G_i,f_i), \varphi_{i j}, I\}$ is an inverse system of finite polyadic groups. This means that for every pair $(i, j)$ of elements of $I$ with $j\leq i$, we are given a polyadic homomorphism
$$
\varphi_{i j}:(G_i, f_i)\to (G_j, f_j)
$$
such that the equality $\varphi_{j k}\varphi_{i j}=\varphi_{i k}$ holds for all $k\leq j\leq i$. Now, assume that
$$
(G, f)=\varprojlim_i (G_i, f_i).
$$
Then $(G, f)$ is called a profinite polyadic group.  From now on, we consider the pair $(G, f)$ which is the above mentioned inverse limit. A realization of this pair can be given as follows: Let $\prod_i(G_i, f_i)$ be the direct product of the family $\{ (G_i, f)_i\}_{i\in I}$. This is a polyadic group with the $n$-ary operation
$$
(\prod f_i)((x_{i1}), (x_{i2}), \ldots, (x_{in}))=(f_i(x_{i1}, x_{i2}, \ldots, x_{in}))_{i\in I}.
$$
Here of course, we denoted an arbitrary element of the direct product as sequence $(a_i)_{i\in I}$ or simply $(a_i)$. Now, we have
$$
G=\{ (x_i)_{i\in I}:\ \forall j\leq i\ \varphi_{i j}(x_i)=x_j\},
$$
and hence
$$
f((x_{i1}), (x_{i2}), \ldots, (x_{in}))=(f_i(x_{i1}, x_{i2}, \ldots, x_{in}))_{i\in I}.
$$

Note that, as each $G_i$ is finite, being a closed subspace of the direct product of a family of finite sets, $(G, f)$ is compact, Hausdorff, and totally disconnected topological polyadic group, of course, if it has been shown that $G$ is non-empty.  Indeed, using standard topological arguments, we can prove that $G\neq\emptyset$ as every $G_i$ is compact.

Recall that, according to Hossz\'{u} -Gloskin's theorem, we have $(G_i, f_i)=\Gfi$, for some ordinary group $(G_i, \bullet_i)$, an element $b_i\in G_i$, and an automorphism $\theta_i$, satisfying the conclusions of Theorem 2.1. We will prove that in some sense, there exists a binary operation $\bullet$ on $G$ such that
$$
(G, \bullet)=\varprojlim_i(G_i, \bullet_i),
$$
and hence $(G, \bullet)$ will be proved to be profinite. Consider the polyadic homomorphism $\varphi_{i j}$. According to Theorem 2.3, there exist an element $a_{i j}\in G_j$, and a  group homomorphism $\psi_{i j}:(G_i, \bullet_i)\to (G_j, \bullet_j)$, such that
$$
\varphi_{i j}=R(a_{i j})\psi_{i j}.
$$
Further, we have the following equalities:\\

$1.\ f_j(a_{i j}, a_{i j}, \ldots, a_{i j})=\psi_{i j}(b_i)\bullet_j a_{i j}$,\\

$2.\ \psi_{i j}\theta_i=I(a_{i j}^{-1})\theta_j\psi_{i j}$.\\

For any triple of indices $k\leq j\leq i$, we have
$$
\varphi_{i j}=R(a_{i j})\psi_{i j}, \ \varphi_{i k}=R(a_{i k})\psi_{i k}, \ \varphi_{j k }=R(a_{j k })\psi_{j k },
$$
therefore
$$
a_{i j}=\varphi_{i j}(1), \ a_{i k}=\varphi_{i k}(1), \ a_{j k }=\varphi_{j k}(1).
$$
Note that in each equality, $1$ is the identity element of the corresponding group. Since $\varphi_{i k}=\varphi_{j k}\varphi_{i j}$, so we have
$$
a_{i k}=\varphi_{j k}(a_{ ij}).
$$
Now, let $Y_i$ be the set of all sequences $(x_j)$ (in the direct product) such that for any $j$ and $k$ $\leq i$, we have $\varphi_{j k}(x_j)=x_k$. This set is non-empty, because we can consider a sequence where $x_j=a_{i j}$, for $j\leq i$, and for all other $j$, $x_j$ is arbitrary. This sequence will be an element of $Y_i$. The set $Y_i$ is closed and if $i\leq s$, then $Y_s\subseteq Y_i$. As the direct product is compact, and the family $\{ Y_i\}$ has finite intersection property, we have
$$
\bigcap_iY_i\neq \emptyset,
$$
showing that $G$ is not empty.

Our first result, shows that the property of being profinite is inherited by the retract and Post's cover:\\

\begin{proposition}
Let $(G, f)=\Gf$ be a profinite polyadic group. Then the retract  group $(G, \bullet)$ and the Post's cover are also profinite.
\end{proposition}

\begin{proof}
We know that $(G, \bullet)=\mathrm{ret}_a(G, f)$ for some $a$. As we have
$$
x\bullet y=f(x,\stackrel{(n-2)}{a},y),
$$
and
$$
x^{-1}=f(\overline{a},\stackrel{(n-3)}{x},\overline{x},\overline{a}),
$$
so, the group $(G, \bullet)$ is a topological group which is compact, Hausdorff, and totally disconnected. This shows that the retract is profinite.
We know that there exists a normal subgroup $K$ of the Post's cover which has the index $n-1$ and $K$ is isomorphic to the retract $(G, \bullet)$. Hence, $K$ is profinite. Now, being a finite extension of a profinite group, $G^{\ast}$ is also profinite.
\end{proof}

The next result shows that, in some sense, the retract is the inverse limit of the retracts of the polyadic groups $(G_i, f_i)$.

\begin{proposition}
Let $(G, f)=\Gf$ be  a polyadic group and $(G, f)=\varprojlim_i (G_i, f_i)$. Then there exist elements $v_i\in G_i$, such that
$$
(G, \bullet)=\varprojlim_i\mathrm{ret}_{v_i}(G_i, f_i).
$$
\end{proposition}

\begin{proof}
As $G\neq \emptyset$, we choose an arbitrary element $(v_i)\in G$. We know that all retracts of a polyadic group are isomorphic to each other. So, we consider the retract
$$
(G_i, \bullet_i)=\mathrm{ret}_{v_i}(G_i, f_i).
$$
By the construction of Sokolov (see \cite{Sok}), we have
$$
\theta_i(x)=f_i(\overline{v}_i, x, \stackrel{(n-2)}{v_i}),
$$
for any $x$. Also we have
$$
b_i=f_i(\overline{v}_i, \ldots, \overline{v}_i).
$$
Using this special form of the retract, we see that the maps $\varphi_{i j}$ are group homomorphisms as well, because
\begin{eqnarray*}
\varphi_{i j}(x\bullet_i y)&=&\varphi_{i j}(f_i(x, \stackrel{(n-2)}{v_i}, y))\\
                           &=&f_j(\varphi_{i j}(x), \stackrel{(n-2)}{\varphi_{i j}(v_i)}, \varphi_{ ij}(y))\\
                           &=&f_j(\varphi_{i j}(x), \stackrel{(n-2)}{v_j}, \varphi_{ ij}(y))\\
                           &=&\varphi_{i j}(x)\bullet_j \varphi_{i j}(y).
\end{eqnarray*}
Note that, here we use the fact $\varphi_{i j}(v_i)=v_j$ as we assumed that $(v_i)\in G$. This shows that the maps $\varphi_{i j}:G_i\to G_j$ are in the same time, group homomorphisms and $\{ (G_i, \bullet_i), \varphi_{i j}, I\}$ is an inverse system of finite groups. Obviously, $(G, \bullet)$ is the inverse limit of this system.
\end{proof}

Note that in some sense, the inverse of the above theorem is also true: if we consider a profinite group $(G, \bullet)$ together with a continuous automorphism $\theta$ and an element $b$ satisfying the requirement of  2.1, then the polyadic group $\Gf$ will be profinite. We will  see the proof soon.  One may ask also about the automorphism $\theta$ in the above proof. The above construction shows that, for any $(x_i)\in G$, we have
$$
\theta((x_i)_{i\in I})=(\theta_i(x_i))_{i\in I}.
$$
As a result, we see that the inverse limit commutes with $\mathrm{der}$:

\begin{corollary}
Let $\{ (G_i,f_i), \varphi_{i j}, I\}$ be an inverse system of finite polyadic groups and for any $i$, suppose $(G_i, f_i)=\Gfi$. Then
$$
\varprojlim_i\Gfi=\mathrm{der}_{\hat{\theta}, \hat{b}}\varprojlim_i(G_i, \bullet_i),
$$
where
$$
\hat{b}=(b_i)_{i\in I},
$$
and
$$
\hat{\theta}((x_i)_{i\in I})=(\theta_i(x_i))_{i\in I}.
$$
\end{corollary}

We are ready now, to give a characterization of the profinite polyadic groups.

\begin{theorem}
 A topological  polyadic group $(G, f)$ is profinite, if and only if, it is compact, Hausdorff, and totally disconnected.
\end{theorem}

\begin{proof}
We already have seen that a profinite polyadic group is compact, Hausdorff and totally disconnected. Now, assume that $(G, f)=\Gf$ is a topological polyadic group which is compact, Hausdorff, and totally disconnected. As we saw before, the retract $(G, \bullet)$ is also a topological group with the same properties, so it is a profinite group. This means that
$$
(G, \bullet)=\varprojlim \{ \frac{G}{K}:\ K\unlhd G, \ K=open\}.
$$
Let
$$
I=\{ K\unlhd G:\ K=open, \ \theta(K)\subseteq K\}, \ J=\{ K\unlhd G:\ K=open\}.
$$
We show that $I$ is a cofinal in the directed set $J$. In other words, we show that for any $L\in J$, there exists a $K\in I$ such that $K\subseteq L$. So, let $L\in J$ and consider
$$
K=\bigcap_{i=0}^{n-1}\theta^i(L).
$$
As $G$ is compact and $\theta$ is continuous, $K$ is an open normal subgroup of $(G, \bullet)$. It is also $\theta$-invariant, since if $u\in K$, then $$
\forall i \exists x_i\in L:\ u=\theta^i(x_i).
$$
Hence
$$
\theta(u)=\theta^{i+1}(x_i)\in \theta^{i+1}(L),
$$
and since $\theta^{n-1}(L)\subseteq L$, so $\theta(u)\in K$. This proves that
$$
(G, \bullet)=\varprojlim_{K\in I}\frac{G}{K}.
$$
Now, for each $K\in  I$, we can define an automorphism
$$
\theta_K:\frac{G}{K}\to \frac{G}{K}, \ \theta_K(xK)=\theta(x)K.
$$
Note that we also have
$$
\theta_K(bK)=bK, \ \theta_K^{n-1}(xK)=(bK)(xK)(bK)^{-1},
$$
therefore, we can consider the finite polyadic group
$$
(\frac{G}{K}, f_K)=\mathrm{der}_{\theta_K, bK}(\frac{G}{K}, \bullet).
$$
Now, as we have
$$
(G, \bullet)=\varprojlim_{K\in I}\frac{G}{K},
$$
the mapping $x\mapsto (xK)_{K\in I}$ is an isomorphism, so every element in the right hand side can be represented as a
sequence $(xK)_{K\in I}$, for a unique $x\in G$. This means that the automorphism
$$
\hat{\theta}((xK)_K)=(\theta(x)K)_K,
$$
is the same as $\theta$. Similarly, we have $\hat{b}=b$. Therefore, by the previous corollary, we have
\begin{eqnarray*}
(G, f)&=& \Gf\\
      &=&\mathrm{der}_{\theta, b}(\varprojlim_{K\in I}(\frac{G}{K}, \bullet))\\
      &=&\varprojlim_{K\in I}\mathrm{der}_{\theta_K, bK}(\frac{G}{K}, \bullet),
\end{eqnarray*}
and this shows that $(G, f)$ is profinite.
\end{proof}

As a result, we have

\begin{corollary}
A polyadic group $(G, f)=\Gf$ is profinite, if and only if, $(G, \bullet)$ is profinite and $\theta$ is continuous.
\end{corollary}

Now, we consider a more general case, where a class $\mathfrak{Y}$ of finite polyadic groups is given and discuss pro-$\mathfrak{Y}$ polyadic groups: polyadic groups which are the inverse limits of  inverse systems of elements of $\mathfrak{Y}$. Let $\mathfrak{X}$ be an arbitrary class of groups. Define a new class
$$
\mathrm{Pol}_n(\mathfrak{X})=\{ (G, f)=n-ary:\ \mathrm{ret}(G, f)\in \mathfrak{X}\}.
$$

\begin{proposition}
If the class $\mathfrak{X}$ is closed under each of the closure operators: subgroup,  direct product, quotient, or  subdirect product, then the class $\mathrm{Pol}_n(\mathfrak{X})$ is also closed under the similar operation.
\end{proposition}

\begin{proof}
Let $\mathfrak{X}$ be closed under subgroup and $(G, f)\in \mathrm{Pol}_n(\mathfrak{X})$. Let $(H, f)\leq (G, f)$. Then for any $a\in H$, we have 
$$
\mathrm{ret}_a(H, f)\leq \mathrm{ret}_a(G, f)\in \mathfrak{X},
$$
and so $\mathrm{ret}_a(H, f)\in \mathfrak{X}$, which shows that $(H, f)\in \mathrm{Pol}_n(\mathfrak{X})$. 

Now, assume that $\mathfrak{X}$ is closed under direct product and $(G_i, f_i)\in \mathrm{Pol}_n(\mathfrak{X})$ be a family of polyadic groups, where $i\in I$. For any arbitrary sequence $(a_i)_i\in \prod_i(G_i, f_i)$ we have 
$$
\mathrm{ret}_{(a_i)}(\prod_i(G_i, f_i))=\prod_i\mathrm{ret}_{a_i}(G_i, f_i)\in \mathfrak{X}, 
$$
and this shows that $\prod_i(G_i, f_i)\in \mathrm{Pol}_n(\mathfrak{X})$. 

Let $\mathfrak{X}$ be closed under quotient, $(G, f)\in \mathrm{Pol}_n(\mathfrak{X})$, and $R$ be a congruence of $(G, f)$. In the quotient polyadic group $(G/R, f_R)$, we have 
$$
f_R([x_1]_R, \ldots, [x_n]_R)=[f(x_1, \ldots, x_n)]_R, 
$$
where $[\ ]_R$ denotes the congruent class. Let $a\in G$ and define a map 
$$
\lambda: \mathrm{ret}_a(G, f)\to \mathrm{ret}_{[a]}(\frac{G}{R}, f_R)
$$
by $\lambda(x)=[x]_R$. We have 
\begin{eqnarray*}
\lambda(x\bullet y)&=&[x\bullet y]_R\\
                   &=&[f(x, \stackrel{(n-2)}{a}, y)]_R\\
                   &=&f_R([x], \stackrel{(n-2)}{[a]}, [y])\\
                   &=&[x]_R\bullet [y]_R.
\end{eqnarray*}
This shows that $\lambda$ is a group epimorphism and hence 
$$
\mathrm{ret}_{[a]}(\frac{G}{R}, f_R)\cong \frac{\mathrm{ret}_a(G, f)}{\ker \lambda}\in \mathfrak{X}. 
$$
As a result, $(G/R, f_R)\in \mathrm{Pol}_n(\mathfrak{X})$. The case of subdirect product can be proved by a similar argument.
\end{proof}

Therefore $\mathrm{Pol}_n(\mathfrak{X})$ will be a variety, if we begin with a variety of groups $\mathfrak{X}$. It could be a good question if one ask about the set of identities of this variety. We are not interested in such problems in this work.  Recall that a pseudo-variety of finite groups, is a class of finite groups which is closed under subgroup, quotient and finite direct product. Similarly, a formation of finite groups, is a class of finite groups which is closed under quotient and finite subdirect products. Now, we are ready to prove the next result. In what follows, a polyadic $\mathfrak{X}$-group means a polyadic $\mathrm{Pol}_n(\mathfrak{X})$-group. So, the name $\mathfrak{X}$ will be used both for the class of groups and its corresponding class of polyadic groups. 

\begin{theorem}
Let $\mathfrak{X}$ be a pseudo-variety (formation) of finite groups. Let $(G, f)$ be a pro-$\mathfrak{X}$ polyadic group and $R$ be an open congruence of it. Then
$$
(\frac{G}{R}, f_R)\in \mathrm{Pol}_n(\mathfrak{X}).
$$
\end{theorem}

\begin{proof}
Here we only consider the case where $\mathfrak{X}$ is a pseudo-variety. Assume that $(G, f)=\Gf$. Suppose $R\subseteq G\times G$ is an open congruence of the polyadic group $(G, f)$. In \cite{Khod-Shah2}, it is  proved that in this case, the equivalence relation $R$ is a subgroup of the ordinary group $G\times G$ (the letter $G$ here stands for the group $(G, \bullet)$). So, we define a map
$$
\psi:\frac{G}{R}\to \frac{G\times G}{R}
$$
by $\psi([x]_R)=(x,1)R$. This map is well-defined as if we suppose $[x]_R=[y]_R$, then $(x,y)\in R$ and so $(x,y)R=R$. This means that
$$
(x,1)R=(1, y^{-1})R,
$$
and as in the quotient we have $(1, y^{-1})R=(y, 1)R$, so the map is well-defined. Also, it is injective, since if $(x,1)R=(y,1)R$, then $(x^{-1}y,1)\in R$, so, we have also
$$
(y, x)=(x,x)(x^{-1}y,1)\in R.
$$
Note that 
\begin{eqnarray*}
\psi(f_R([x_1], \ldots, [x_n]))&=&\psi([f(x_1, \ldots, x_n)]_R)\\
                               &=&(f(x_1, \ldots, x_n), 1)R\\
                               &=&(x_1, 1)(\theta(x_2), 1)\cdots (\theta^{n-1}(x_n), 1)(b, 1)R\\
                               &=&f_{\frac{G\times G}{R}}((x_1, 1)R, \ldots, (x_n, 1)R).
\end{eqnarray*}
This shows that we indeed have a polyadic embedding 
$$
\psi: (\frac{G}{R}, f_R)\to \mathrm{der}_{\bar{\theta}, \bar{b}}(\frac{G\times G}{R}, \bullet). 
$$
As $(G, f)$ is assumed to be a pro-$\mathfrak{X}$ polyadic group, $(G, \bullet)$ is a pro-$\mathfrak{X}$ group, and hence $G\times G$ is so. Now, $R\unlhd G\times G$ is an open subgroup, so the quotient $G\times G/R$ belongs to $\mathfrak{X}$. This means that
$$
\mathrm{der}_{\bar{\theta}, \bar{b}}(\frac{G\times G}{R}, \bullet)\in \mathrm{Pol}_n(\mathfrak{X}),
$$
and hence $(G/R, f_R)\in \mathrm{Pol}_n(\mathfrak{X})$.  
\end{proof}

The converse of the above theorem is also true. For a proof, we need to proceed as in the theorem 3.4. 

\begin{theorem} 
Let $\mathfrak{X}$ be a  pseudo-variety (formation) of finite groups. Let $(G, f)$ be a topological polyadic group which is compact, Hausdorff, and totally disconnected. Assume that for any open congruence $R$, the polyadic group $(G/R, f_R)$ belongs to $\mathrm{Pol}_n(\mathfrak{X})$. Then $(G, f)$ is pro-$\mathfrak{X}$. 
\end{theorem}

\end{document}